\documentclass[a4paper,10pt]{amsart}
\usepackage[arrow,matrix]{xy}
\usepackage{amsmath,amssymb,amscd,bbm,amsthm,mathrsfs,dsfont,bm}
\usepackage{extarrows}
\usepackage{latexsym}
\usepackage{graphicx}
\usepackage{color}
\allowdisplaybreaks
\newtheorem{theorem}{Theorem}[section]
\newtheorem{lemma}[theorem]{Lemma}
\newtheorem{proposition}[theorem]{Proposition}
\newtheorem{corollary}[theorem]{Corollary }
\theoremstyle{definition}
\newtheorem{definition}[theorem]{Definition}

\newtheorem{remark}[theorem]{Remark}

\newtheorem{hypothesis}[theorem]{Hypothesis}
\numberwithin{equation}{theorem}

\newtheorem{example}[theorem]{Example}
 
 \DeclareMathOperator{\Ext}{Ext}
 
\DeclareMathOperator{\Hom}{Hom}

\DeclareMathOperator{\gr}{gr}

\DeclareMathOperator{\rees}{Rees}
\DeclareMathOperator{\LH}{LH}
\def\bt{\begin{theorem}}
\def\et{\end{theorem}}
\def\bl{\begin{lemma}}
\def\el{\end{lemma}}
\def\br{\begin{remark}}
\def\er{\end{remark}}
\def\bc{\begin{corollary}}
\def\ec{\end{corollary}}

\begin{document}
\title[Note on Nakayama automorphisms of PBW deformations and Hopf actions]
{Note on Nakayama automorphisms of\\ PBW deformations and Hopf actions
}

\author{Y. Shen}
\address{Department of Mathematics, Zhejiang University,
Hangzhou 310027, China}
\email{11206007@zju.edu.cn}

\author{D.-M. Lu}
\address{Department of Mathematics, Zhejiang University,
Hangzhou 310027, China}
\email{dmlu@zju.edu.cn}







\begin{abstract}
PBW deformations of Artin-Schelter regular algebras are skew Calabi-Yau. We prove that the Nakayama automorphisms of such PBW deformations can be obtained from their homogenizations. Some Calabi-Yau properties are generalized without Koszul assumption. We also show that the Nakayama automorphisms of such PBW deformations control Hopf actions on them.
\end{abstract}

\maketitle

\section*{Introduction}
Skew Calabi-Yau algebras, as a generalization of Calabi-Yau algebras, attract lots of researchers to study. The latest paper \cite{RRZ1} introduced a definition of graded skew Calabi-Yau category. Paralleling with Calabi-Yau algebras, \cite[Theorem 3.5]{RRZ1}  presented that certain full triangulated subcategories of derived categories of a large class of algebras, which is called generalized Artin-Schelter Gorenstein algebras, are graded skew Calabi-Yau categories. Such algebras possess a kind of specific automorphisms called the Nakayama automorphisms. This is one of significant features of skew Calabi-Yau algebras.  When restricted on connected graded case, skew Calabi-Yau algebras are just Artin-Schelter regular algebras. The recent papers \cite{CWZ,LMZ} showed the that Nakayama automorphisms control Hopf actions on Artin-Schelter regular algebras.

Since the Nakayama automorphism is a well behaved invariant, how to calculate becomes a primary issue. Fortunately, there are many homological identities to help obtain the Nakayama automorphisms of noetherian Artin-Schelter regular algebras (see \cite{RRZ}, for example). In \cite{LWW}, the authors acquire the Nakayama automorphisms under Ore extensions. In fact, the Nakayama automorphisms of noetherian algebras are related to rigid dualizing complexes. Using the particular dualizing complexes, the paper \cite[Proposition 6.18]{YZ} proved that an algebra whose associated graded algebra is connected noetherian  Artin-Schelter regular for some filtration is endowed with a Nakayama automorphism induced by the one of associated Rees algebra. PBW deformations of noetherian Artin-Schelter regular algebras satisfy those conditions, such as the Weyl algebras and the ungraded Down-Up algebras. The paper \cite{G} told that such PBW deformations are skew Calabi-Yau algebras. In the case of Koszul and low dimension, paper \cite{HVZ} has described the Nakayama automorphisms of such PBW deformations explicitly. Those Nakayama automorphisms are induced by the ones of the Rees algebras, but the Rees algebras are not easy to deal with in general. Our plan from the beginning was to seek a handy method for calculating the Nakayama automorphisms of such PBW deformations.

We realize that homogenizations of PBW deformations will be one of the options that meet our requirements. By using the homogenizations, we give a different way to show, without noetherian assumption, that a PBW deformation of an Artin-Schelter regular algebra is skew Calabi-Yau with the Nakayama automorphism induced from the homogenization (see Theorem \ref{nakayama automorphism of pbw}). The proof does not involve rigid dualizing complexes and localizations,  and it is different from the method for the Rees algebras. This result is available and effective for computation in practice.  Moreover, using the homogenization approach, we generalize some Calabi-Yau properties handily in Corollary \ref{CY properties}.

In recent years, Hopf actions on low dimensional Artin-Schelter regular algebras have been made progress (see \cite{CWZ,LMZ}). An extensive question is Hopf actions on filtered Artin-Schelter regular algebras. The case of $2$-dimension has been done in \cite{CWWZ}. However, it becomes a tough and tedious work even for $3$-dimensional case. Analogous to \cite{LMZ}, we give some conditions to characterize Hopf algebras which act on PBW deformations of noetherian Artin-Schelter regular algebras. It implies their Nakayama automorphisms also control Hopf actions. We achieve it also with the help of homogenization.

Here is an outline of the paper. In Section 1, we recall definitions of skew Calabi-Yau algebras and Artin-Schelter regular algebras and review PBW deformations in briefly. Section 2 is devoted to calculating Nakayama automorphisms of PBW deformations of Artin-Schelter regular algebras. We reprove that the Weyl algebras are Calabi-Yau algebras and obtain the Nakayama automorphisms of ungraded Down-Up algebras and some other examples. We also generalize some Calabi-Yau properties without Koszul hypothesis. We mainly consider Hopf actions and give some conditions to determine whether Hopf algebras are group algebras in Section 3.

Throughout the paper, $k$ is an algebraically closed field of characteristic $0$.  All algebras and Hopf algebras are over $k$. Unless otherwise stated, tensor product $\otimes$ means $\otimes_k$.

\section{Preliminaries}
Let $A$ be an algebra, and let $A^e$ be the enveloping algebra $A\otimes A^o$ where $A^o$ is the opposite algebra of $A$. Let $M$ be a left (resp. right) $A^e$-module, so it is a left and right $A$-module (resp. $A^o$-module). A left $A^e$-module and a right $A^e$-module can be identified because of $A^e\cong (A^e)^o$. For two automorphisms $\mu,\nu$ of $A$, the \emph{twisted module} $^\mu M^\nu$ is defined such that it is just $M$ as $k$-vector spaces, and the module structure becomes $a\ast m\ast b=\mu(a)m\nu(b)$ for any $a,b\in A$ and $m\in M$.

Suppose $A$ is also a graded algebra, that is, $A=\bigoplus_{i\in\mathbb{Z}}A_i$. For any $i\in\mathbb{Z}$, each element in $A_i$ is called \emph{homogeneous element} with degree $i$. We say $A$ is \emph{connected} if $A_i=0$ for all $i<0$ and $A_0=k$.  Let $M$ be a left graded $A$-module. For some integer $i$, \emph{shift} of $M$ by degree $i$ is $M(i)$, defined by $M(i)_j=M_{i+j}$ for any $j$. Let $\sigma$ be a graded automorphism, then \emph{graded twisted algebra} $A^{\sigma}$ is defined such that $A^{\sigma}=A$ as vector spaces, and the multiplication satisfying $a\star b=\sigma^{\deg b}(a)b$.

\begin{definition}
An algebra $A$ is called \emph{skew Calabi-Yau} (\emph{skew CY}, for short) if
\begin{enumerate}
\item $A$ is homologically smooth, that is, $A$ has a finite length projective resolution as left $A^e$-module such that each term is finitely generated;
\item there exists an automorphism $\mu$ of $A$ such that
\begin{equation}\label{skew CY iso}
\Ext^i_{A^e}(A,A^e)\cong
\left\{
\begin{array}{cl}
^1A^\mu&\text{if}\; \;i=d,\\
\;0&\text{if}\;\; i\neq d
\end{array}
\right.
\end{equation}
as left $A^e$-modules, where $1$ denotes the identity map of $A$.
\end{enumerate}
\end{definition}
\begin{remark}
\begin{enumerate}
\item If an algebra $A$ satisfies (\ref{skew CY iso}), then we say $A$ has a \emph{Nakayama automorphism} $\mu$ (always denoted by $\mu_A$). It is unique up to inner automorphisms of $A$ if it exists.

\item Suppose an algebra $A$ has a Nakayama automorphism $\mu_A$. The proof of \cite[Corollary 3.6]{YZ} still works and implies $\mu_A(z)=z$ for any central element $z\in A$.

\item An algebra $A$ is \emph{Calabi-Yau} (CY, for short) if and only if $A$ is skew CY and the identity map is a Nakayama automorphism.

\item If $A$ is a graded algebra, then the definition of skew CY algebra should be in the category of graded $A^e$-modules, and (\ref{skew CY iso}) should be replaced by
\begin{equation*}
\Ext^i_{A^e}(A,A^e)\cong
\left\{
\begin{array}{cl}
^1A^\mu(l)&\text{if}\; \;i=d,\\
\;0&\text{if}\;\; i\neq d
\end{array}
\right.
\end{equation*}
for some $l\in \mathbb{Z}$. The Nakayama automorphism $\mu$ is also a graded isomorphism.
\end{enumerate}
\end{remark}

\begin{definition}
A connected graded algebra $A$ is \emph{Artin-Schelter Gorenstein} (AS-Gorenstein, for short) of dimension $d$ if the following conditions hold:
\begin{enumerate}
\item $A$ has finite injective dimension $d$;
\item $\Ext^d_A(k,A)=k(l)$ and $\Ext^i_A(k,A)=0$ if $i\neq d$ for some $l\in
\mathbb{Z}$. The integer $l$ is called \emph{Gorenstein parameter}.
\item the right version of (b) holds.
\end{enumerate}

In addition, $A$ has finite global dimension $d$, then $A$ is  \emph{Artin-Schelter regular} (AS-regular, for short)
\end{definition}

\begin{theorem}\label{AS-regular iff skew CY}\cite[Lemma 1.2]{RRZ}
Let $A$ be a connected graded algebra. Then $A$ is AS-regular if and only if $A$ is skew CY.
\end{theorem}

The theorem shows that AS-regular algebras are naturally endowed with Nakayama automorphisms. Actually, any noetherian AS-Gorenstein algebras have Nakayama automorphisms which can be characterized by balanced dualizing complexes. Van den Bergh introduced rigid dualizing complexes in \cite{VDB} especially for ungraded noetherian algebras. It is easy to check that a noetherian algebra $A$ has finite injective dimension with Nakayama automorphism $\mu_A$ if and only if $A$ has a rigid dualizing complex $^1\!A\,^{\mu_A^{-1}}[d]$ for some integer $d$ (consider $^1\!A\,^{\mu_A^{-1}}$ as a complex concentrated in degree $0$. The symbol $[\;\; ]$ means a shift of complex).

Now we review the definition of PBW deformations. Let $A=k\langle X\rangle/(R)$ be a connected graded $k$-algebra generated by degree $1$, where $X=\{x_1,x_2,\cdots,x_n\}$ and $R=\{r_1,r_2,\cdots,r_s\}$. Let $X^*$ be the free monoid generated by $X$ including $1$. Define a new algebra $U=k\langle X\rangle/(P)$, where
$$
P=\{r_i+r'_i\,|\,r'_i\in k\langle X\rangle \text{ and }\deg r'_i<\deg r_i,\ i=1,\cdots,s\}.
$$

There is a canonical filtration on $U$ induced by degree (or length) of monomials in $X^*$. Let $$F_i\,U=\frac{F_ik\langle X\rangle+(P)}{(P)}\,,$$ where $F_ik\langle X\rangle=\text{Span}_k\{u\in X^*\,|\,\deg u\leq i\}$ if $i\geq0$, and $F_i\,U=0$ if $i<0$. The associated graded algebra $\operatorname{gr}U=\bigoplus_{i\in\mathbb{Z}} F_i\,U/ F_{i-1} U$, and the associated Rees algebra $\rees_FU=\bigoplus_{i\geq0}F_iUt^i$ . It is well known $\gr U\cong\rees_FU/(t)$ and $U\cong\rees_FU/(1-t)$. And we have a natural graded surjective homomorphism
$$
\phi: A\to \operatorname{gr}U.
$$

\begin{definition}
Retain the notations above. We say $U$ is a \emph{PBW deformation} of $A$ if $\phi$ is a graded isomorphism.
\end{definition}

Papers \cite{BG,FV0} presented the necessary and sufficient conditions on whether an algebra is a PBW deformation of Koszul algebra and of $N$-Koszul algebra respectively. In general, we can describe it by Gr\"obner basis. Let $x_1<x_2<\cdots<x_n$. Fix the ordering $<$ on $X^*$ to be deg-lex order, that is, $u<v$ if $\deg u<\deg v$ or $\deg u=\deg v$ and there exist factorizations $u=rx_is, v=rx_jt$ such that $i<j$ where $r,s,t\in X^*$ for any $u,v\in X^*$. Let $f=\sum_{i=1}^s f_i\in k\langle X\rangle$ where each nonzero $f_i$ is a homogeneous polynomial and $\deg f_1<\deg f_2<\cdots<\deg f_s$, we call $f_s$ the \emph{leading homogeneous polynomial} of $f$, which is denoted by $\LH(f)$. And for any set $S$ of polynomials, $\LH(S)=\{\LH(f)\,|\,f\in S\}$. Then $\gr U\cong k\langle X\rangle/(\operatorname{LH}(G))$ by \cite[Chapter 4, Theorem 2.3]{LHS}, where $G$ is the Gr\"obner basis of $(P)$ with respect to $<$. In other words, $U$ is a PBW deformation of $A$ if and only if $\LH(G)$ is a Gr\"obner basis of $(R)$ where $G$ is the Gr\"obner basis of $(P)$.

The main property we needed for PBW deformations is the following result.

\begin{theorem}\label{pbw deformation is skew CY}\cite[Corollary 3.5]{G}
Let $A$ be a connected graded noetherian skew CY algebra, and let $U$ be a PBW deformation of $A$. Then $U$ is a skew CY algebra.
\end{theorem}
\begin{remark}
The proof of this theorem in \cite{G} depends on rigid dualizing complexes. Hence it should be restricted on noetherian algebras which is lost in the original statement. We will give another proof in Theorem \ref{nakayama automorphism of pbw} without noetherian assumption.
\end{remark}

 \section{Homogenizations and Nakayama automorphisms}\label{homogenization}
In this section, we introduce the definition of homogenization algebras. Then using the homogenization algebras to compute the Nakayama automorphisms of PBW deformations of AS-regular algebras.

Denote the free algebra $k\langle X\rangle$ by $F$ where $X=\{x_1,x_2,\cdots,x_n\}$ with $\deg x_i=1$.  Consider the polynomial extension $F[t]$ of $F$ with $\deg t=1$. Let $f$ be a polynomial of degree $s$ in $F$, then $f=\sum_{i=0}^s f_i$ where each $f_i$ is a homogeneous polynomial of degree $i$. Now we define a corresponding element $f^\mathbf{t}=\sum_{i=0}^sf_it^{s-i}$ in $F[t]$ which is homogeneous in $F[t]$. We say $f^\mathbf{t}$ is the \emph{homogenization} of $f$ and the central element $t$ is the \emph{homogenization element}.
\begin{definition}
Retain the notations above. Let $U=F/(f_1,\cdots,f_m)$ be an algebra, where $f_1,\cdots,f_m\in F$. The graded algebra $H(U)=F[t]/(f_1^\mathbf{t},\cdots,f^\mathbf{t}_m)$ is called the \emph{homogenization} of $U$.
\end{definition}

It is easy to know $H(U)/(1-t)\cong U$. If $U$ is a PBW deformation of a connected graded algebra $A$, then $H(U)/(t)\cong A\cong\gr U$.  Firstly, we show that the Nakayama automorphisms of some algebras can be deduced form their regular central extension.

\begin{lemma}\label{nakayama automorphism of central}
Let $B$ be an algebra with a Nakayama automorphism $\mu_B$. Suppose $z$ is a regular central element in $B$ such that $B$ is a flat module over $k[z]$. Then $B/(z)$ has a Nakayama automorphism $\mu_{B/(z)}$ equals $\mu_B$ when induced on $B/(z)$.
\end{lemma}
\begin{proof}
Let $d$ be the integer such that $\Ext^d_{B^e}(B,B^e)\cong ^1\!\!\!B^{\mu_B}$. Write $\bar{B}=B/(z)$. There exists a well-defined automorphism $\mu_{\bar{B}}$ on $\bar{B}$ induced by $\mu_B$ because of $\mu_B(z)=z$. We have a free resolution of $\bar{B}$ as left $B$-modules which is also a resolution as right $B$-modules:
$$
P_{\centerdot}: \quad 0\xrightarrow{\ \ } B\xrightarrow{\,z\,} B\xrightarrow{\ \ } \bar{B}\xrightarrow{\ \ }0.
$$
Similarly, denote by $P^o_{\centerdot}$ the analogous resolution of $\bar{B}^o$ as left $B^o$-modules, and also as right $B^o$-modules. Thus $\Ext^i_B(\bar{B},B)=\Ext^i_{B^o}(\bar{B}^o,B^o)=0$ for all $i\neq 1$, $\Ext^1_B(\bar{B},B)\cong \bar{B}$ as left $B^e$-modules, and $\Ext^1_{B^o}(\bar{B}^o,B^o)\cong \bar{B}^o$ as left $(B^o)^e$-modules.

Notice that $P_{\centerdot}\otimes P^o_{\centerdot}$ is a free resolution of $\bar{B}\otimes \bar{B}^o$ as left $B^e$-modules and a free resolution as right $B^e$-modules. So the spectral sequence implies $\Ext^i_{B^e}(\bar{B}^e,B^e)=0$ if $i\neq 2$ and $\Ext^2_{B^e}(\bar{B}^e,B^e)\cong \bar{B}^e$ as left $(B^e)^e=B^e\otimes (B^e)^o$-modules, also as left ($\bar{B}^e)^e$-modules.

Take a free resolution $Q_{\centerdot}$ of $\bar{B}$ as left $\bar{B}^e$-modules and an injective resolution $I^{\centerdot}$ of $B^e$ as left $(B^e)^e$-modules. The complex $Q_{\centerdot}$ is also quasi-isomorphic to $\bar{B}$ as left $B^e$-modules, and $I^{\centerdot}$ is also an injective resolution of $B^e$ when restricted to left $B^e$-modules. For all $i\geq 0$ and as right $B^e$-modules, also as left $B^e$-modules,
\begin{align*}
\Ext^i_{\bar{B}^e}(\bar{B},\bar{B}^e)&\cong H^i(\Hom_{\bar{B}^e}(Q_{\centerdot},\Ext^2_{B^e}(\bar{B}^e,B^e))\\
                               &\cong H^i(\Hom_{\bar{B}^e}(Q_{\centerdot},H^2(\Hom_{B^e}(\bar{B}^e,I^{\centerdot}))))\\
                               &\cong H^{i+2}(\Hom_{\bar{B}^e}(Q_{\centerdot},\!\Hom_{B^e}(\bar{B}^e,I^{\centerdot})))\\
                               &\cong H^{i+2}(\Hom_{B^e}(\bar{B}^e\otimes_{\bar{B}^e}Q_{\centerdot},I^{\centerdot}))\\
                               &\cong \Ext^{i+2}_{B^e}(\bar{B},B^e).
\end{align*}

Write $w=z\otimes 1\in B^e$.  Then $w$ is a regular central element in $B^e$, so $B^e$ is a $k[w]$-algebra. With the natural right $k[w]$-module structure on $B$, it is also a $B^e$-$k[w]$-bimodule.

Let $C_{\centerdot}$ be a free resolution of $B$ as left $B^e$-$k[w]$-bimodules, and take a free resolution of $k$ as $k[w]$-modules as follows,
$$
C'_{\centerdot}:\quad 0\xrightarrow{\ \ } k[w]\xrightarrow{\,w\,} k[w]\xrightarrow{\ \ } k\xrightarrow{\ \ }0.
$$
Tensoring $B$ to $C'_\centerdot$, the regularity of $w$ implies $\bar{B}\cong B\otimes_{k[w]}k$ as left $B^e$-modules.

Since $B$ is a right flat module over $k[z]$, $B$ and $B^e$ are right flat $k[w]$-modules.
Thus complex $D_\centerdot=C_\centerdot\otimes_{k[w]}C'_\centerdot$ is a free resolution of $B\otimes_{k[w]} k$ as left $B^e$-modules.  Applying functor $\Hom_{B^e}(-,B^e)$ to $D_\centerdot$, the flatness of $B$ and $B^e$ and the spectral sequence implies as left $B^e$-modules
\begin{align*}
\Ext^{d+1}_{B^e}(\bar{B},B^e)&\cong \Ext^{d+1}_{B^e\otimes_{k[w]}k[w]}(B\otimes_{k[w]} k,B^e\otimes_{k[w]}k[w])\\
                               &\cong \Ext^{d}_{B^e}(B,B^e)\otimes_{k[w]}\Ext^1_{k[w]}(k,k[w])\\
                               &\cong ^1\!B^{\mu_B}\otimes_{k[w]}k\\
                               &\cong ^1\!\!\bar{B}^{\mu_{\bar{B}}}.
\end{align*}
and $\Ext^{i}_{B^e}(\bar{B},B^e)=0$ if $i\neq d+1$.

Consequently, $\Ext^i_{\bar{B}^e}(\bar{B},\bar{B}^e)=0$ if $i\neq d-1$, and $\Ext^{d-1}_{\bar{B}^e}(\bar{B},\bar{B}^e)\cong ^1\!\!\bar{B}^{\mu_{\bar{B}}}$.
\end{proof}

If $U$ is a PBW deformation of a connected graded algebra $A$, then the homogenization $H(U)$ is a regular central extension of $A$ by \cite[Theorem 1.3]{CS}. That is to say, homogenization element $t$ is regular, so $H(U)$ is a flat $k[t^i]$-module for any $i\geq 1$ by \cite[Chapter 3, Corollary 3.51]{Ro}. So the condition flatness of Lemma \ref{nakayama automorphism of central} is always automatically satisfied for homogenizations.

The main result of this section is the following, which is a new version of Theorem \ref{pbw deformation is skew CY} with no noetherian assumption on the algebra $A$.

\begin{theorem}\label{nakayama automorphism of pbw}
Let $U$ be a PBW deformation of an AS-regular algebra $A=k\langle X\rangle/(R)$, $H(U)$ be the homogenization algebra of $U$. Then $H(U)$ is skew CY with the Nakayama automorphism $\mu_{H(U)}$, and $U$ is skew CY with a Nakayama automorphism $\mu_U$ induced by $\mu_{H(U)}$.
\end{theorem}
\begin{proof}
Write $X=\{x_1,x_2,\cdots,x_n\}$ and $H=H(U)$. Suppose $A$ has global dimension $d$. By \cite[Theorem 1.3]{CS}, $t$ is a central regular element in $H$. So $H$ is AS-regular of dimension $d+1$ by \cite[Lemma 3.5]{ZZ1} and skew CY with Nakayama automorphism $\mu_{H}$ by Theorem \ref{AS-regular iff skew CY}.

By Theorem \ref{AS-regular iff skew CY}, $A$ has a finite length of finitely generated projective resolution as left $A^e$-modules. Since $\gr U\cong A$, $\gr U^e\cong A^e$ for a natural filtration on $U^e$. Thus there is an associated finite length of finitely generated projective resolution of $U$ as left $U^e$-modules by \cite[Chapter 2, Proposition 2.5]{LHS}. Then it remains to show that $U$ has a Nakayama automorphism.

The element $t$ is regular central and $1-t$ is central in $H$ with $\mu_H(1-t)=1-t$. Write  $z=1-t$. For any polynomial $f\in k[z]$ and any nonzero element $h\in H$, it is easy to check that $fh=0$ implies $f=0$. So $H$ is a flat $k[z]$-module by \cite[Chapter 3, Corollary 3.51]{Ro}. Thus $U\cong H/(1-t)$ also has a Nakayama automorphism $\mu_U$ induced by $\mu_H$ from Lemma \ref{nakayama automorphism of central}.

More precisely, if $\mu_{H}(x_i)=\sum_{j=1}^na_{ij}x_j+b_it$ for some $a_{ij},b_{i}\in k$, then $\mu_U(x_i)=\sum_{j=1}^na_{ij}x_j+b_i$ for all $i=1,\cdots,n$.
\end{proof}

Immediately, we have a CY property about $U$.
\begin{corollary}\label{homogenization is CY then PBW deformation is CY}
Let $U$ be a PBW deformation of an AS-regular algebra $A$, and let $H(U)$ be the homogenization of $U$. If $H(U)$ is CY, then $U$ and $A$ are CY.
 \end{corollary}

The following lemma is proved for $N$-Koszul algebras in \cite{HVZ}, we give a general situation.
\begin{lemma}\label{nakayama automorphism of gr}
Let $U$ be a PBW deformation of an  AS-regular algebra $A=k\langle X\rangle/(R)$ with the Nakayama automorphism $\mu_A$. Then there exists a Nakayama automorphism $\mu_U$ of $U$ preserving the canonical filtration on $U$ and $\gr(\mu_U)=\mu_A$.

In addition, if $A$ is domain then the filtration-preserving Nakayama automorphism $\mu_U$ is unique.
\end{lemma}
\begin{proof}
Write $X=\{x_1,x_2,\cdots,x_n\}$ and $\mu_A(x_i)=\sum_{j=1}^na_{ij}x_j$ for some $a_{ij}\in k$. Let $H$ be the homogenization of $U$ with the Nakayama automorphism of $\mu_H$. Since $H/(t)\cong A$ and Lemma \ref{nakayama automorphism of central}, we have $\mu_H(x_i)=\sum_{j=1}^na_{ij}x_j+b_it$ for some $b_i\in k$. Now using Theorem \ref{nakayama automorphism of pbw}, $\mu_U(x_i)=\sum_{j=1}^na_{ij}x_j+b_i$. It is a filtration preserving automorphism and $\gr (\mu_U)=\mu_A$.

Assume $A$ is domain, so is $U$. Suppose $\mu_U$ and $\mu'_U$ are two Nakayama automorphisms of $U$ preserving the filtration such that $\gr(\mu_U)=\gr(\mu'_U)=\mu_A$. However, the Nakayama automorphism is unique up to inner isomorphisms, there exists an invertible elements $w$ in $U$ such  that $\mu_U(x_i)=w\mu'_U(x_i)w^{-1}$. Consider the image of $w\mu'_U(x_i)w^{-1}$ in $A$ should be nonzero, it forces $w,w^{-1}\in F_0U= k$. Thus $\mu_U=\mu'_U$.
\end{proof}

Although the filtration-preserving Nakayama automorphism of PBW deformation $U$ is not unique in general, we always choose the Nakyama automorphism $\mu_U$ of $U$ as in Theorem \ref{nakayama automorphism of pbw} in the sequel. In this case, when the Nakayama automorphism $\mu_H$ of homogenization $H$ applies to the generators of $H$, $\mu_H(x_i)$ has the form $\mu_U(x_i)^{\mathbf{t}}$ for every $x_i\in X$.

\begin{example}\label{pbw of quantum plane}
Let quantum plane $k_q[x_1,x_2]=k\langle x_1,x_2\rangle/(x_1x_2-qx_2x_1)$ where $q$ is a scalar in $k^{\times}$ with the Nakayama automorphism sending
$x_1\mapsto qx_1$ and $x_2\mapsto q^{-1}x_2.$ We consider PBW deformations of it. Assume $q\neq1$. By definition,
$$
U=k\langle x_1,x_2\rangle/(x_1x_2-qx_2x_1+ax_1+bx_2+c)
$$
are all PBW deformations where $a,b,c\in k$. Let $H$ be the homogenization of $U$. Then $k_q[x_1,x_2]\cong H/(t)$. So by Lemma \ref{nakayama automorphism of central}, the Nakayama automorphism $\mu_H$ of $H$ is
$$
\mu_H(x)=qx_1+d_1t,\quad\mu_H(x_2)=q^{-1}x_2+d_2t,
$$
for some $d_1,d_2\in k$. By Theorem \ref{nakayama automorphism of pbw}, $U$ has a Nakayama automorphism $\mu_U$ satisfying
$$
\mu_U(x_1)=qx_1+d_1,\quad\mu_U(x_2)=q^{-1}x_2+d_2.
$$
However, $\mu_U$ is a well-defined algebra homomorphism, so it needs $d_1=-b$ and $d_2=q^{-1}a$. Thus $U$ has a Nakayama automorphism sending $x_1\mapsto qx_1-b$ and $x_2\mapsto q^{-1}x_2+q^{-1}a$.
\end{example}

In the following, we will find some examples whose Nakayama automorphisms of  AS-regular algebras and ones of PBW deformations of them have the same form. Here are some general conditions.

\begin{proposition}\label{form of nakayama automorphism}
Let $X=\{x_1,x_2,\cdots,x_n\}$ and let $U$ be a PBW deformation of an AS-regular algebra $A=k\langle X\rangle/(R)$. Let $H$ be the homogenization of $U$. Assume the Nakayama automorphism $\mu_A$ of $A$ such that $\mu_A(x_i)=\sum_{j=1}^n a_{ij}x_j$ for some $a_{ij}\in k$ and all $i=1,2,\cdots,n$.

\begin{enumerate}
\item If $H/(t^2)\cong A[t]/(t^2)$, then there exists a Nakyama automorphism $\mu_U$ of $U$ such that $\mu_U(x_i)=\sum_{j=1}^n a_{ij}x_j$ for any $i=1,\cdots,n$.
\item Let $J$ be the Jordan canonical form of the matrix $M=(a_{ij})$. Suppose $M$ has no eigenvalue $1$, then there exists a basis $\{1,v_1,v_2,\cdots,v_n\}$ of $F_1U=k\oplus k\{X\}$ such that $\mu_U(\mathrm{\mathbf{v}})=J\mathrm{\mathbf{v}}$  where the column vector $\mathrm{\mathbf{v}}=(v_1,v_2,\cdots,v_n)^T$.
\end{enumerate}
\end{proposition}
\begin{proof}
By Lemma \ref{nakayama automorphism of central} and Theorem \ref{nakayama automorphism of pbw}, the Nakayama automorphism $\mu_H$ of $H$ satisfies $\mu_H(x_i)=\sum_{j=1}^n a_{ij}x_j+b_it$ and there exists a Nakayama automorphism $\mu_U$ of $U$ such that $\mu_U(x_i)=\sum_{j=1}^n a_{ij}x_j+b_i$ where $b_i\in k$ and $i=1,\cdots,n$.

(a) By Lemma \ref{nakayama automorphism of central}, we have
$$
\mu_H=\mu_{H/(t^2)}=\mu_{A[t]/(t^2)}=\mu_{A[t]}
$$
when all maps are induced on $H/(t^2)$. Thus $\mu_H(x_i)=\sum_{j=1}^n a_{ij}x_j$. So $b_i=0$ for all $i=1,\cdots,n$.

(b) Write the matrix $N=\left(\begin{array}{cc}1&0\\\mathrm{\mathbf{b}}&M\end{array}\right)$, which is the corresponding matrix of $\mu_H$ with respect to the basis $\{t,x_1,\cdots,x_n\}$, where $\mathrm{\mathbf{b}}=(b_1,b_2,\cdots,b_n)^T$ . If $M$ has no eigenvalue $1$, then $N$ has the Jordan canonical form $\operatorname{diag}(1,J)$. The conclusion follows.
\end{proof}

It has been proved that the Weyl algebras are CY in \cite{Be}. Also \cite{CWWZ,HZ,LWW} gave different approaches to the conclusion. However, the proof of \cite{Be,CWWZ} used the Koszul dual. The authors of \cite{HZ} defined a kind of special quadratic algebras and studied the CY property of their PBW deformations which contains the Weyl algebras. A more general way to obtain the Nakayama automorphisms of the Weyl algebras is using Ore extension (see \cite[Remark 4.2]{LWW}). Here we use homogenization to calculate the Nakayama automorphisms of the Weyl algebras.

\begin{corollary}
Weyl algebra $A_n(k)$ is CY.
\end{corollary}
\begin{proof}
Weyl algebra $A_n(k)$ is a PBW deformation of $B$, where $B=k[x_1,\cdots,x_n,y_1,\cdots,y_n]$. Let $H$ be the homogenization of $A_n(k)$. It is clear that $H/(t^2)\cong B[t]/(t^2)$. However the Nakayama automorphism of $B$ is identity, so $A_n(k)$ is CY by Proposition \ref{form of nakayama automorphism}(a).
\end{proof}

\begin{example}\label{down-up}
The Down-Up algebra $A(\alpha,\beta,\gamma)=k\langle x_1,x_2\rangle/(f_1,f_2)$ was introduced in \cite{BR} firstly, where
\begin{eqnarray*}
&&f_1=x_1^2x_2-\alpha x_1x_2x_1-\beta x_2x_1^2-\gamma x_1,\\
&&f_2=x_1x_2^2-\alpha x_2x_1x_2-\beta x_2^2x_1-\gamma x_2,
\end{eqnarray*}
and $\alpha,\beta,\gamma\in k$. We assume $\beta\neq 0$, so $A(\alpha,\beta,\gamma)$ is noetherian and $A(\alpha,\beta,0)$ is $3$-dimensional AS-regular (\cite{KMP}). One can easily check that $A(\alpha,\beta,\gamma)$ is a PBW deformation of $A(\alpha,\beta,0)$.  Write $B=A(\alpha,\beta,0)[t]$. By Example $A(6)$ in \cite{LMZ}, the Nakayama automorphism $\mu_B$ of $B$ satisfies
$$
\mu_B(x_1)=-\beta x_1,\quad \mu_B(x_2)=-\beta^{-1}x_2,\quad \mu_B(t)=t.
$$

Let $H$ be the homogenization of $A(\alpha,\beta,\gamma)$. Then $H/(t^2)\cong B/(t^2)$. Now by Proposition \ref{form of nakayama automorphism}(a), we know the Nakayama automorphism of $A(\alpha,\beta,\gamma)$ is
$$
x_1\mapsto-\beta x_1,\quad x_2\mapsto-\beta^{-1}x_2.
$$
Moreover, $A(\alpha,\beta,\gamma)$ is CY if and only if $\beta=-1$.
\end{example}

Following we investigate examples with non-diagonalizable Nakayama automorphisms.

\begin{example}\label{pbw of hpbw of down-up}
Starting with graded Down-Up algebras, there are four classes of Artin-Schelter regular algebras which are \emph{homogeneous PBW deformations} of graded Down-Up algebras (see \cite[Example 5.3]{SZL}). The following are two special cases:  $\mathfrak{B}=k\langle x_1,x_2\rangle/(f_1,f_2)$ where
\begin{eqnarray*}
&&f_1 =x_1^2x_2-x_2x_1^2 -x_2x_1x_2+x_2^2x_1   ,\\
&&f_2 = x_1x_2^2 - x_2^2x_1,
\end{eqnarray*}
and $\mathfrak{C}=k\langle x_1,x_2\rangle/(g_1,g_2)$  where
\begin{eqnarray*}
&&g_1 =,x_1^2x_2-2x_1x_2x_1+x_2x_1^2 -x_2x_1x_2+x_2^2x_1\\
&&g_2 = x_1x_2^2-2x_2x_1x_2 + x_2^2x_1.
\end{eqnarray*}
The element $z=x_1x_2-x_2x_1$ is normal in both $\mathfrak{B}$ and $\mathfrak{C}$. Then $$\mathfrak{B}/(z)\cong\mathfrak{C}/(z)\cong k[x_1,x_2].$$
By \cite[Lemma 1.5]{RRZ}, we have the Nakayama automorphism of $\mathfrak{B}$ sending
$
x_1\mapsto -x_1-x_2,\; x_2\mapsto -x_2,
$
and the Nakayama automorphism of $\mathfrak{C}$ sending
$
x_1\mapsto x_1+x_2,\; x_2\mapsto x_2.
$

Let $U_1=k\langle x_1,x_2\rangle/(f_1,f_2+z)$ and $U_2=k\langle x_1,x_2\rangle/(g_1,g_2+z)$. They are two PBW deformations of $\mathfrak{B}$ and $\mathfrak{C}$ respectively.  Let $H_1,H_2$ be the homogenizations of $U_1,U_2$ respectively. Then $z$ is still a normal element in $H_1$ and $H_2$. Clearly,
$$
H_1/(z)\cong H_2/(z)\cong k[x_1,x_2,t].
$$

By \cite[Lemma 1.5]{RRZ} and Theorem \ref{nakayama automorphism of pbw}, we obtain a Nakayama automorphism of $U_1$ sending
$$
x_1\mapsto -x_1-x_2-1,\quad x_2\mapsto -x_2-1,
$$
and a Nakayama automorphism of $U_2$ sending
$$
x_1\mapsto x_1+x_2+1,\quad x_2\mapsto x_2+1.
$$
\end{example}

We conclude this section by utilizing homogenization to a CY properties to general setting. In the proof, there is a homological identity involving the notation \emph{homological determinant}. We only give a brief introduction to the definition. For the details, we refer to \cite[Section 3]{KKZ} and \cite[Section 3]{RRZ}.

Let $A$ be a noetherian AS-Gorenstein algebra, and let $K$ be a Hopf algebra. Suppose $A$ is a left $K$-module algebra. The balanced dualizing complex of $A$ is $R\cong\,^{\mu}\!A^1(-l)[d]$ where $d$ is injective dimension of $A$, $l$ is Gorenstein parameter and $\mu$ is the Nakayama automorphism of $A$. Then there is a left $K$-module structure on degree $(-d)$ part of $R$, so on $^{\mu}\!A^1(-l)$. Since the degree $l$ piece of $^{\mu}\!A^1(-l)$ is dimension one, each element of $K$ acts on it by a scalar. Then we use this point to define an algebra homomorphism $\textsf{hdet}: K\to k$ which is called \emph{homological determinant}. On the other hand, assume $K$ is finite dimensional, then $\textsf{hdet}$ is an element in the finite dual $K^o$. Then it is called the \emph{homological codeterminant} of the $K^o$-coaction on $A$ from right. If $\textsf{hdet}$ is just the counit of $K$, we say it is \emph{trivial}. For example, Hopf algebra $K$ is just a group algebra where the group is a subgroup of automorphism group of $A$, so every automorphism has an image in $k$ through $\textsf{hdet}$.

If $A$ is a graded algebra, then denote by $\xi_c$ an endomorphism of $A$ sending $a$ to $c^{\deg a}a$ for some $c\in k^{\times}$ and all $a\in A$ . With Koszul hypothesis, the next results have been proved in \cite[Theorem 3.1(1,3)]{WZ} and  \cite[Theorem 0.1 and 0.2]{HVZ1}.
\begin{corollary}\label{CY properties}
Let $A=k\langle X\rangle/(R)$ be an AS-regular algebra with Nakayama automorphism $\mu_A$. Let $U$ be a PBW deformation of $A$, and $H$ be the homogenization of $U$.
\begin{enumerate}
\item If $A$ is a domain, then $H$ is CY if and only if $U$ is CY.
\item If $A$ is noetherian. Choose the Nakayama automorphism $\mu_U$ of $U$ induced by the Nakayama automorphism $\mu_H$ of $H$. Then the skew extensions $A[z,\mu_A]$ and $U[z,\mu_U]$ are CY.
\end{enumerate}
\end{corollary}
\begin{proof}
Set $X=\{x_1,x_2,\cdots,x_n\}$, $R=\{r_1,r_2,\cdots,r_m\}$ and $U=k\langle  X\rangle/(P)$, where $P=\{\tilde{r}_1,\tilde{r}_2,\cdots,\tilde{r}_m\}$.

(a) It follows from Corollary \ref{homogenization is CY then PBW deformation is CY} and
 Lemma \ref{nakayama automorphism of gr} immediately.

(b) The first part can be deduced from \cite[Corollary 0.6(1)]{RRZ} and \cite[Theorem 5.3]{RRZ1}. For the completeness, we prove it based on $\mathbb{Z}$-graded twist here.

Consider the graded twisted algebras $B=A^{\mu\!_A}[z]$ and $C=B^{\sigma}$ where $\sigma$ is an automorphism of $B$ such that $\sigma_{|_{A^{\mu\!_A}}}=\mu_A^{-1}$ and $\sigma(z)=z$. Then $C\cong A[z,\mu_A]$.  If the Gorenstein parameter of $A$ is $l$, then the one of $B$ and of $C$ are $l+1$. Let $\mu_C$ be the Nakayama automorphism of $C$. Since $A$ is noetherian AS-regular, $\textsf{hdet} \mu_A=\textsf{hdet} \mu^{-1}_A=1$ by \cite[Theorem 5.3]{RRZ1}. Thus $\textsf{hdet}\sigma=1$ by \cite[Lemma 6.1]{RRZ}. Using the homological identity of \cite[Theorem 5.4(a)]{RRZ}, we have
$$
\mu_B(z)=z, \quad\mu_B(a)=\mu_{A}\circ\mu_A^l\circ\xi^{-1}_{\textsf{hdet}\mu\!_A}(a)=\mu_A^{l+1}(a),
$$
for any $a\in A$. Then
\begin{eqnarray*}
&&\mu_{C}(a)=\mu_{B}\circ\sigma^{l+1}\circ\xi^{-1}_{\textsf{hdet}
\sigma}(a)=\mu^{l+1}_A\circ\mu^{-(l+1)}_A(a)=a,\;\text{ for any }a\in A,\\
&&\mu_{C}(z)=\mu_{B}\circ\sigma^{l+1}\circ\xi^{-1}_{\textsf{hdet}\sigma}(z)=z.
\end{eqnarray*}
Thus, we have $\mu_C=\operatorname{id}_C$ and $A[z,\mu_A]$ is CY.

Now turn to the case of PBW deformation. Let $H$ be the homogenization of $U$, and let $H(D)$ be the homogenization of $D=U[z,\mu_U]$ whose generators are $x_1,x_2,\cdots,x_n,z,t$ with the relations
\begin{eqnarray*}
&\tilde{r}_1^\mathbf{t},\quad\tilde{r}_2^\mathbf{t},\quad\cdots,\quad \tilde{r}_n^\mathbf{t},\\
&x_it-tx_i,\; zt-tz,\; zx_i-\mu_U(x_i)^\mathbf{t}z.
\end{eqnarray*}

Since $\mu_H(x_i)$ has the form $\mu_U(x_i)^\mathbf{t}$ by Theorem \ref{nakayama automorphism of pbw}. So
$H(D)\cong H[z,\mu_H]$. By the last assertion, $H(D)$ is CY. Notice that $D$ is also a PBW deformation of AS-regular algebra $A[z,\mu_A]$, it follows that $D=U[z,\mu_U]$ is CY by Corollary \ref{homogenization is CY then PBW deformation is CY}.
\end{proof}

\section{Hopf actions on PBW deformations}
In this section, we consider finite dimensional Hopf algebras acting on PBW deformations of noetherian AS-regular algebras. Our main idea is transferring this question to their homogenizations. For the facts of Hopf algebras, we refer to \cite{Mo}.

In the sequel, we always assume $X=\{x_1,x_2,\cdots,x_n\}$, $A=k\langle X\rangle/(R)$ is a noetherian AS-regular algebra generated by degree $1$, $U=k\langle X\rangle/(P)$ is a PBW deformation of $A$, $H$ is the homogenization of $U$ and $\mu_A,\mu_U,\mu_H$ are the Nakayama automorphisms of $A,U,H$ respectively. Let
$K$ be a finite dimensional Hopf algebra. There is a standard hypothesis for the rest of this section.

\begin{hypothesis}\label{hopf action on pbw}
We assume that $U$ is a left $K$-module algebra, $K$-action on $U$ preserves the canonical filtration, and $K$ acts on $U$ \emph{inner faithfully}, namely, there is no nonzero  Hopf ideal $I\subset H$ such that $IU=0$.
\end{hypothesis}

We say that a right  $K^o$-coaction on a $K^o$-comodule $T$ is \emph{inner faithful}, if for any proper Hopf subalgebra $K'$ of $K^o$ such that $\rho(T) \not\subset T\otimes K'$. Let $T$ be a left $K$-module, then $K$-action on $T$ is inner faithful if and only if the right $K^o$-coaction on $T$ is inner faithful (\cite[Lemma 1.3(a)]{CWZ}).

\begin{definition}
Let $K$ be a Hopf algebra, and $B$ an algebra. Suppose $K$ acts on $B$, then the \emph{fixed subalgebra} of $B$ defined by
$$
B^K=\{b\in B\,|\, g\cdot b=\epsilon(g)b \,\text{ for any } g\in K \}.
$$
\end{definition}

Under the Hypothesis \ref{hopf action on pbw}, it is clear that there exists a $K$-action on $A$ naturally preserving grading. However, the $K$-action on $A$ may be not inner faithful. There is an additional condition for the $K$-action on $U$ in \cite{CWWZ}, which is automatically satisfied if $K$ is semisimple, to guarantee the inner faithfulness. Our approach is to give a natural $K$-action on $H$ and this action inherits the inner faithfulness.

Suppose $K$-action on $U$ satisfying $g\cdot x_i=\sum_{j=1}^nc^g_{ij}x_j+d^g_i$ where $g\in K$ and $c^g_{ij},d^g_i\in k$. Then define a $K$-action on $H$ by, for any  $g\in K$
$$
g\triangleright x_i:=\sum_{j=1}^nc^g_{ij}x_j+d^g_it,\qquad
g\triangleright t:=\epsilon(g)t.
$$
We distinguish the $K$-action on $U$ and $H$ by $\cdot$ and $\triangleright$. Since $g\triangleright f=(g\cdot f)^\mathbf{t}$ for any homogeneous $f\in k\langle X\rangle$, the $K$-action on $H$ is well-defined. It can be checked straightforward that $K$ acts on $H$ inner faithfully by definition. Conversely, any inner faithful Hopf action on $H$ preserves grading such that $t\in H^K$ can induce an inner faithful Hopf action on $U$ naturally. Hence, the problems of Hopf action on $U$ are equivalent to the problems of Hopf action on $H$ satisfying $t\in H^K$. We say $K$-action on $U$ has a \emph{trivial homological determinant} if $K$-action on $H$ has a trivial homological determinant.

The following result may be considered a version of filtered AS-Gorensteiness of the fixed subalgebras of PBW deformations.
\begin{lemma}
Let $H(U^K)$ be the homogenization of $U^K$. Then $H^K\cong H(U^K)$.
\end{lemma}
\begin{proof}
Notice that $t\in H^K$ and $g\triangleright f^\mathbf{t}=(g\cdot f)^\mathbf{t}$ for any $f\in k\langle X\rangle,g\in K$. Others are clear.
\end{proof}

We are interested in the question that when a Hopf algebra $K$ acting on $U$ is semisimple or a group algebra. In \cite{LMZ}, the authors gave lots of conditions to determine what  finite dimensional Hopf algebras that act on noetherian AS-regular algebras are. Here, we have some analogues results to judge the same questions on PBW deformations of noetherian AS-regular algebras.

Let $\{1,v_1,v_2,\cdots,v_n\}$ be a basis of $F_1 U=k\oplus k\{X\}$. There exists a right $K^o$-coaction $\rho_U:U\to U\otimes K^o$ on $U$. Since the $K$-action preserves the filtration, then
$$
\rho_U(v_i)=\sum_{j=1}^n v_j\otimes y_{ji}+1\otimes z_i,\quad\text{ for all }i=1,2,\cdots,n,
$$
for some $y_{ij},z_i\in K^o$. It also induces  $K^o$-coaction on $H$ satisfying
\begin{eqnarray*}
&&\rho_H(v_i^\mathbf{t})=\sum_{j=1}^n v_j^\mathbf{t}\otimes y_{ji}+t\otimes z_i,\quad\text{ for all }i=1,2,\cdots,n,\\
&&\rho_H(t)=t\otimes 1.
\end{eqnarray*}

By coassociation, $\bigtriangleup(y_{ij})=\sum_{s=1}^ny_
{is}\otimes y_{sj}$ and $\bigtriangleup(z_i)=\sum_{s=1}^n z_s\otimes y_{si}+1\otimes z_i$, $\epsilon(y_{ij})=\delta_{ij}$ and $\epsilon(z_i)=0$. The $K^o$-coaction on $U$ and on $H$ are inner faithful, and the set $\{y_{ij},z_i\}$
generates $K$ as Hopf algebra by \cite[Lemma 1.3(b)]{CWZ}.

As the argument above Example \ref{pbw of quantum plane}, we choose the Nakayama automorphism $\mu_U$ of $U$ induced by  the one of homogenization $H$. For convenience,  suppose $\mu_U(x_i)=\sum_{j=1}^n a_{ij}x_j+b_i$ and $\mu_H(x_i)=\sum_{j=1}^n a_{ij}x_j+b_it$  where $a_{ij},b_i\in k$ for any $i=1,\cdots,n$. Set the invertible matrix $M=(a_{ij})$, and
\begin{equation*}
N=\left(
\begin{array}{cc}
1&0\\
\mathrm{\mathbf{b}} & M
\end{array}
\right)
\end{equation*}
where $\mathrm{\mathbf{b}}$ is the column vector $(b_1,b_2,\cdots,b_n)^T$. Let $J_M$ be the Jordan canonical form of $M$.

\subsection{Diagonalizable}Firstly, we consider $J_M$ is a diagonal matrix. Assume that $M$ has no eigenvalue $1$. By Proposition \ref{form of nakayama automorphism}, then $N$ has the Jordan canonical form $\operatorname{diag} (1,J_M)$.

\begin{lemma}\label{diagonal hopf action}
Retain the notations above. Assume $J_M=\operatorname{diag}(\lambda_1,\lambda_2,\cdots,\lambda_n)$ with $\lambda_i\neq 1$ for all $i=1,\cdots,n$. Choose the basis $\{1,v_1,\cdots,v_n\}$ of $F_1U$ such that $\mu_U(v_i)=\lambda_i v_i$ for $i=1,\cdots,n$.  Suppose either
\begin{enumerate}
\item[(1)] $\lambda_i\lambda_j^{-1}$ is not a root of unity for all $i\neq j$ and $\lambda_i$ is not a root of unity for all $i$. or

\item[(2)]$K$ is semisimple, $K$-action on $U$ has trivial homological determinant (or $K^o$-coaction has trivial homological codeterminant) and $\lambda_i\neq \lambda_j$ for all $i\neq j$.
\end{enumerate}
Then $y_{ij}=z_i=0$ for all $i\neq j$. As a consequence, $K$ is a finite dual of a group algebra and $K$ is semisimple. In addition,
\begin{enumerate}
\item If $U$ has a relation of form
$$
r=u_1v_iv_ju_2-s_0u_1u_2+s_1w_1+s_2w_2+\cdots+s_{m}w_{m},
$$
where $s_0,s_1,\cdots,s_m\in k^\times$, $u_1,u_2,w_1,w_2\cdots,w_m\in X^*$ and $m$ is non-negative integer for some  $i\neq j$. Assume that $\{1,w_1,w_2,\cdots,w_{m}\}$ are $k$-linearly independent. Then $y_{ii}=y_{jj}^{-1}$.
\item If $U$ has relations of form
$$
r_{ij}=u^{ij}_1v_iv_ju^{ij}_2-s^{ij}_0u^{ij}_1v_jv_iu^{ij}_2+s^{ij}_1w^{ij}_1+s^{ij}_2w^{ij}_2+\cdots+s^{ij}_{m_{ij}}w^{ij}_{m_{ij}},
$$
where $s^{ij}_0,s^{ij}_1,\cdots,s^{ij}_{m_{ij}}\in k^\times$, $u^{ij}_1,u^{ij}_2,w^{ij}_1,w^{ij}_2\cdots,w^{ij}_{m_{ij}}\in X^*$ and $m_{ij}$ is non-negative integer for all  $i<j$. Assume that $\{u^{ij}_1v_jv_iu^{ij}_2,w^{ij}_1,w^{ij}_2,\cdots,w^{ij}_{m_{ij}}\}$ are $k$-linearly independent for all $i<j$. Then both $K$ and $K^o$ are commutative group algebras.
\end{enumerate}
\end{lemma}
\begin{proof}
The corresponding matrix $N$ of $\mu_H$ has eigenvalues $\{1,\lambda_1,\cdots,\lambda_n\}$. Lifting $v_i$ to $H$, then $\mu_H(v_i^\mathbf{t})=\lambda_iv_i^\mathbf{t}$. By \cite[Lemma 2.3(1,2,4,5)]{LMZ}, $y_{ij}=z_i=0$ for all $i\neq j$. So each $y_{ii}$ is a group-like element, and $\rho_U(v_i)=v_i\otimes y_{ii}$ for all $i$.  Since $\{y_{ii}\}_{i=1}^n$ generates $K^o$ as Hopf algebra by inner faithfulness, $K$ is a finite dual of group algebra. As a consequence, $K$ is semisimple.

(a) Assume $u_1=x_{h_1}^{p_1}x_{h_2}^{p_2}\cdots x_{h_l}^{p_l}$ and $u_2=x_{f_1}^{q_1}x_{f_2}^{q_2}\cdots x_{f_{l'}}^{q_{l'}}$. Applying $\rho_U$ to $r$, we have
\begin{align*}
0=\rho_U(r)&=u_1v_iv_ju_2\otimes y_{h_1h_1}^{p_1}y_{h_2h_2}^{p_2}\cdots y_{h_lh_l}^{p_l}y_{ii}y_{jj}y_{f_1f_1}^{q_1}y_{f_2f_2}^{q_2}\cdots y_{f_{l'}f_{l'}}^{q_{l'}}\\
& \qquad-s_0u_1u_2\otimes y_{h_1h_1}^{p_1}y_{h_2h_2}^{p_2}\cdots y_{h_lh_l}^{p_l}y_{f_1f_1}^{q_1}y_{f_2f_2}^{q_2}\cdots y_{f_{l'}f_{l'}}^{q_{l'}}+\cdots\\
&=s_0u_1u_2\otimes y_{h_1h_1}^{p_1}y_{h_2h_2}^{p_2}\cdots y_{h_lh_l}^{p_l}(y_{ii}y_{jj}-1)y_{f_1f_1}^{q_1}y_{f_2f_2}^{q_2}\cdots y_{f_{l'}f_{l'}}^{q_{l'}}\\
&\qquad+s_1w_1\otimes g_1+s_2w_2\otimes g_2+\cdots+s_mw_m\otimes g_m,
\end{align*}
for some $g_1,g_2,\cdots,g_m\in K^o$. By the $k$-linearly independent, $y_{ii}y_{jj}=1$.

(b) Similar to the proof of (a). Apply $\rho_{U}$ to $r_{ij}$ for all $i<j$, then we get $y_{ii}$ commutes with $y_{jj}$. In other words, $K^o$ is commutative and cocommutative. so $K$ is a commutative group algebra.
\end{proof}

As an application, we consider the Hopf actions on PBW deformations of quantum plane. Partial result has been proved in \cite[Corollary 5.8(a)]{CWWZ}.

\begin{corollary}\label{Hopf action on pbw of quantum plane}
Let $U=k\langle x_1,x_2\rangle/(f)$ be a PBW deformation of quantum plane $k_q[x_1,x_2]$ with $q\neq 1$, where $f=x_1x_2-qx_2x_1+ax_1+bx_2+c$. Let $K$ act on $U$ satisfying Hypothesis \ref{hopf action on pbw}. Then
\begin{enumerate}
\item If $q$ is not a root of unity, then $K$ is a commutative group algebra.
\item Suppose $K$ is semisimple and $K$-action on $U$ has trivial homological determinant. If $q\neq -1$, then $K$ is a commutative group algebra.\end{enumerate}

In the setting of \rm{(a)} or \rm{(b)}, if further $c\neq abq(1-q)^{-1}$, we have $K\cong k\mathbb{Z}_m$ for some $m$.
\end{corollary}
\begin{proof}
Since $q\neq1$, $U$ is isomorphic to
$$
k\langle x_1,x_2\rangle/(x_1x_2-qx_2x_1-1),\quad \text{ or } \; k_q[x_1,x_2].
$$
So $\mu_U(x_1)=qx_1$ and $\mu_U=q^{-1}x_2$ by Example \ref{pbw of quantum plane}. Now (a) and (b) are obtained from Lemma \ref{diagonal hopf action} immediately. The right $K^o$-coaction on $U$ should be
$$
\rho_U(x_1)=x_1\otimes y_{11},\quad \rho_U(x_2)=x_2\otimes y_{22},
$$
for some commuting elements $y_{11},y_{22}\in K^o$.

If $c\neq abq(1-q)^{-1}$, then $U\cong k\langle x_1,x_2\rangle/(x_1x_2-qx_2x_1-1)$. By Lemma \ref{diagonal hopf action}(b), we have $y_{11}=y_{22}^{-1}$. So $K^o=kG$, where $G$ is a group generated by $y_{11}$. Then $K$ is also a group algebra of cyclic group. More precisely,  $g\cdot x_1=\xi x_1,\; g\cdot x_2=\xi^{-1} x_2$ where $\xi$ is a root of unity, if $K=k\langle g\rangle$.
\end{proof}

\begin{corollary}\label{Hopf action on pbw of down-up}
Let $K$ act on Down-Up algebra $A(\alpha,\beta,\gamma)$ where $\beta\neq 0$ satisfying Hypothesis \ref{hopf action on pbw}.
\begin{enumerate}
\item If $\beta$ is not a root of unity, then $K$ is a dual of a group algebra.
\item If $\alpha\neq0$ and $\beta$ is not a root of unity, then $K$ is a commutative group algebra.
\item If $\gamma\neq0$ and $\beta$ is not a root of unity, then $K\cong k\mathbb{Z}_m$ for some $m$.
\item Suppose $K$ is semisimple and $K$-action on $A(\alpha,\beta,\gamma)$ has trivial homological determinant. If $\beta\neq \pm1$, then $K$ is a dual of a group algebra. In addition, \begin{enumerate}
\item[(\romannumeral1)]if $\alpha\neq0$, then $K$ is a commutative group algebra.
\item[(\romannumeral2)]if $\gamma\neq 0$, then $K\cong k\mathbb{Z}_m$ for some $m$.
\end{enumerate}
\end{enumerate}
\end{corollary}
\begin{proof}
Let $U=A(\alpha,\beta,\gamma)$. By Example \ref{down-up}, we have that $U$ is a PBW deformation of $A=A(\alpha,\beta,0)$ and $\mu_U(x)=-\beta x_1,\mu_U(x_2)=-\beta^{-1}x_2$. Now (a) is obtained from Lemma \ref{diagonal hopf action} immediately. The right $K^o$-coaction should be
$$
\rho_U(x_1)=x_1\otimes y_{11},\quad \rho_U(x_2)=x_2\otimes y_{22},
$$
for some elements $y_{11},y_{22}\in K^o$.

(b, c) If $\alpha\neq 0$, there exists a relation of $R$ satisfying Lemma \ref{diagonal hopf action}(b). So $K$ is a commutative group algebra.

If $\gamma\neq0$, then Lemma \ref{diagonal hopf action}(a) applies. Thus $K^o$ is a cyclic group algebra, and $K\cong k\mathbb{Z}_m$.  More precisely,  $g\cdot x_1=\xi x_1,\; g\cdot x_2=\xi^{-1} x_2$ where $\xi$ is a root of unity, if $K=k\langle g\rangle$.

(d) It is similar to (a, b, c) by Lemma \ref{diagonal hopf action}.
\end{proof}

\subsection{Non-diagonalizable}Next, we consider the non-diagonalizable case of $J_M$.
\begin{lemma}\label{non-diagonal hopf action}
Suppose $A$ is generated by two generators $\{x_1,x_2\}$.
\begin{enumerate}
\item[(1)] The Nakayama automorphism $\mu_U$ satisfies
$
x_1\mapsto \lambda x_1+b_1,\; x_2\mapsto b_2+ax_1+\lambda x_2,
$
for some $\lambda,a\in k^\times$, $b_1,b_2\in k$ and $\lambda\neq 1$. Suppose $K$ is semisimple and $K$-action on $U$ has trivial homological determinant.  Then $K\cong k\mathbb{Z}_m$ for some $m$.

\item[(2)] The Nakayama automorphism $\mu_U$ satisfies
$
x_1\mapsto x_1+b_1,\; x_2\mapsto b_2+ax_1+x_2
$
for some $a,b_1\in k^\times$ and $b_2\in k$. Then $K$ is trivial ($K\cong k$).
\end{enumerate}
\end{lemma}
\begin{proof}
(1) By Proposition \ref{form of nakayama automorphism} and $\lambda\neq1$, there exists a basis $\{1,v_1,v_2\}$ of $F_1U$ such that $\mu_U(v_1)=\lambda v_1$ and $\mu_U(v_2)=\lambda v_2+v_1$.

Let $H$ be the homogenization of $U$. The matrix corresponding to the Nakayama automorphism $\mu_H$ of $H$ with respect to the basis $\{t,v_1^\mathbf{t},v_2^\mathbf{t}\}$ is
$$N=\left(\begin{array}{ccc}1&0&0\\0&\lambda&0\\0&1&\lambda\end{array}\right).$$

If $K$ is semisimple and $K$-action has trivial homological determinant, the conditions of \cite[Lemma 3.4(2)]{LMZ} are satisfied. So $y_{ij}=z_i=0$ for all $i\neq j$ and $y_{11}=y_{22}$. Then $K\cong \mathbb{Z}_m$.

More precisely, $g\cdot v_1=\xi v_1,\; g\cdot v_2=\xi v_2$ where $\xi$ is a root of unity, if $K=k\langle g\rangle$.

(2) There exists a basis $\{1,v_1,v_2\}$ of $F_1U$ such that $\mu_U(v_1)=v_1+c$ and $\mu_U(v_2)=v_2+cv_1+c^2$ for some $c\in k^\times$. Now consider the $K$-action on homogenization of $U$. By \cite[Lemma 3.4(3)]{LMZ}, $y_{ij}=z_i=0$ for all $i\neq j$ and $y_{11}=y_{22}=1$. Thus $K\cong k$.
\end{proof}
\begin{corollary}
Let $U_1$ and $U_2$ be ones defined in Example \ref{pbw of hpbw of down-up} and let finite dimensional Hopf algebras $K_1$ and $K_2$ act on $U_1$ and $U_2$ satisfying Hypothesis \ref{hopf action on pbw} respectively. Suppose $K_1$ is semisimple and $K_1$-action on $U_1$ has trivial homological determinant. Then $K_1\cong k\mathbb{Z}_m$ for some $m$ and $K_2$ is trivial.
\end{corollary}
\begin{proof}
It follows from Examples \ref{pbw of hpbw of down-up} and Lemma \ref{non-diagonal hopf action}.
\end{proof}

\vskip 7mm
{\it Acknowledgments.}  This research is supported by the NSFC
(Grant No. 11271319). We thank the reviewers for careful reading and valuable suggestions.


\vskip10mm

\begin{thebibliography}{10}
\bibitem{Be}
R. Berger,  \emph{Gerasimov's theorem and $N$-Koszul algebras}, J. London Math. Soc., \textbf{79} (2009), 631--648.

\bibitem{BR}
G. Benkart and T. Roby, \emph{Down-Up algebras}, J. Algebra, \textbf{209} (1998), 305--344.

\bibitem{BG}
A. Braverman and D. Gaitsgory, \emph{Poincar\'e-Birkhoff-Witt Theorem for Quadratic Algebras of Koszul Type}, J. Algebra, \textbf{181} (1996), 315--328.

\bibitem{CS}
T. Cassidy and B. Shelton, \emph{PBW-deformation theory and regular central extensions}, J. Reine Angew. Math., \textbf{610} (2007), 1--12.

\bibitem{CWWZ}
K. Chan, C. Walton, Y. H. Wang and J. J. Zhang, \emph{Hopf actions on filtered regular algebras}, J. Algebra, \textbf{397} (2014), 68--90.

\bibitem{CWZ}
K. Chan, C. Walton and J. J. Zhang, \emph{Hopf actions and Nakayama automorphisms}, J. Algebra, \textbf{409} (2014), 26--53.

\bibitem{FV0}
G. Fl\o ystad and J. E. Vatne, \emph{PBW-deformations of $N$-Koszul algebras}, J. Algebra, \textbf{302(1)} (2006), 116--155.

\bibitem{G}
J. Gaddis, \emph{PBW deformations of Artin-Schelter regular algebras}, preprint arXiv:1210.0861v2, (2012).


\bibitem{HVZ}
J. W. He, F. Van Oystaeyen and Y. H. Zhang, \emph{Deformations of Koszul Artin-Schelter Gorenstein algebras}, Manuscripta Math., \textbf{141} (2013), 463--483.

\bibitem{HVZ1}
J. W. He, F. Van Oystaeyen and Y. H. Zhang, \emph{Skew polynomial algebras with coefficients in Koszul Artin-Schelter regular algebras}, J. Algebra, \textbf{390} (2013), 231--249.

\bibitem{HZ}
J. W. He and S. H. Zhang, \emph{Deformations of quadratic algebras with antisymmetric relations}, Comm. Algebra, \textbf{39} (2011), 2137--2149.

\bibitem{KKZ}
E. Kirkman, J. Kuzmanovich and J. J. Zhang, \emph{Gorenstein subrings of invariants under Hopf algebra actions}, J. Algebra,  \textbf{322} (2009), 3640--3669.

\bibitem{KMP}
E. Kirkman, I. M. Musson and D. S. Passman, \emph{Noetherian Down-Up algebras}, Proc. Amer. Math. Soc., \textbf{127}(1999), 3161--3167.

\bibitem{LHS}
H.-S. Li, \emph{Gr\"{o}bner bases in ring theory}, Word Scientific Publishing Co. Pte. Ltd., (2012).

\bibitem{LWW}
L.-Y. Liu, S.-Q. Wang and Q.-S. Wu, \emph{Twisted Calabi-Yau property of Ore extensions}, J. Noncommut .Geom., \textbf{8} (2014), 587--609.

\bibitem{LMZ}
J.-F. L\"u, X.-F. Mao and J. J. Zhang, \emph{Nakayama automorphism and applications}, preprint arXiv:1408.5761v2, (2014).

\bibitem{Mo}
S. Montgomery, \emph{Hopf algebras and their actions on rings}, volume 82 of CBMS Regional Conference Series in Mathematics. Published for the Conference Board of the Mathematical Sciences, Washington, DC, (1993).

\bibitem{RRZ}
M. Reyes, D. Rogalski and J. J. Zhang, \emph{Skew Calabi-Yau algebras and homological identities}, Adv. Math., \textbf{264} (2014), 308--354.

\bibitem{RRZ1}
M. Reyes, D. Rogalski and J. J. Zhang, \emph{Skew Calabi-Yau triangulated categories and Frobenius Ext-algebras}, preprint arXiv:1408:0536v2, (2014).

\bibitem{Ro}
J. J. Rotman, \emph{An introduction to homological algebra}, Second edition, Universitext. Springer, New York, (2009).

\bibitem{SZL}
Y. Shen, G.-S. Zhou and D.-M. Lu, \emph{Homogeneous PBW-deformation for Artin-Schelter regular algebras}, Bull. Aust. Math. Soc., \textbf{91} (2015), 53--68.
\bibitem{VDB}
M. Van den Bergh, \emph{Existence theorems for dualizing complexes over non-commutative graded and filtered rings}, J. Algebra, \textbf{195} (1997), 662--679.

\bibitem{WZ}
Q.-S. Wu and C. Zhu, \emph{Poincar\'e-Birkhoff-Witt deformation of Koszul Calabi-Yau algebras}, Algebra Repr. Theory, \textbf{16} (2013), 405--420.

\bibitem{YZ}
A. Yekutieli and J. J. Zhang, \emph{Rings with Auslander dualizing complexes}, J.
Algebra, \textbf{213} (1999), 1--51.

\bibitem{ZZ1}
J. J. Zhang and J. Zhang, \emph{Double Ore extension}, J. Pure Appl. Algebra, \textbf{212} (2008), 2668--2690.
\end{thebibliography}
\end{document}